\documentclass[12pt,draftclsnofoot,onecolumn]{IEEEtran}

\ifCLASSINFOpdf
  \else
\fi
\usepackage{caption}
\usepackage{amsmath}
\usepackage{mathrsfs}
\usepackage{amsthm}
\usepackage{multirow}
\usepackage{cite}
\usepackage{color}
\usepackage{makecell}
\usepackage{subfigure}
\usepackage{graphicx}      % include this line if your document contains figures
\newtheorem{remark}{Remark}
\newtheorem{lemma}{Lemma}

\newtheorem{theorem}{Theorem}

\newcommand{\be}{\begin{eqnarray}}
\newcommand{\ee}{\end{eqnarray}}
\newcommand{\bi}{\begin{itemize}}
\newcommand{\ei}{\end{itemize}}
\newcommand{\ba}{\begin{array}}
\newcommand{\ea}{\end{array}}
\newcommand{\bm}{\begin{matrix}}
\newcommand{\eem}{\end{matrix}}

\newcommand{\no}{\nonumber}

\begin{document}
\title{ Superlinear Optimization Algorithms}

%\author{Huanshui~Zhang,~Hongxia~Wang, Yeming~Xu, Ziyuan~Guo}

\author{
        Hongxia Wang,\IEEEmembership{}
       Yeming Xu,~\IEEEmembership{}% <-this % stops a space
       Ziyuan Guo,\IEEEmembership {}
       Huanshui Zhang$^*$\IEEEmembership{}
\thanks{H. Zhang, H. Wang, Y. Xu, and Z. Guo are with the School
of Electrical and Automation Engineering, Shandong University of Science and Technology, Qingdao 266590, China (e-mail: hszhang@sdu.edu.cn; whx1123@163.com;  ymxu2022@163.com; skdgzy@sdust.edu.cn).}
\thanks{This work was supported by  the Original Exploratory Program Project of National Natural Science Foundation of China (62250056),  the Joint Funds of the National Natural Science Foundation of China (U23A20325),  the Major Basic Research of Natural Science Foundation of Shandong Province (ZR2021ZD14), and the High-level Talent Team Project of Qingdao West Coast New Area (RCTD-JC-2019-05). 	(Corresponding author: Huanshui Zhang).}}% <-this % stops a space
%\thanks{L. Li is with the School of Information Science and Engineering, Shandong Agricultural University, Taian, 271018, China,
%email:(lin\_li@sdau.edu.cn).
%}% <-this % stops a space
%\thanks{M. Fu is with SchoolI of Engineering, University of Newcastle, Callaghan, NSW 2308, Australia (minyue.fu@newcastle.edu.au)}
%}
\maketitle
\begin{abstract}
%In this paper, we investigate the stabilization problem of networked control
%systems with both input delays and packet losses. Different from the existing stabilization results for single-input networked control systems, we focus on the stabilization for multi-input networked control systems, where every control input has diverse measurability. To address this problem, we parameterize the admissable controllers by combining the orthogonal decomposition and predictor feedback techniques. We establish several sets of necessary and sufficient stabilizing conditions for multi-input networked control systems. We provide an alternating optimization approach to acquire the maximum packet-loss rate. Finally, the achieved results are illustrated by numerical examples.

This paper proposes several novel optimization algorithms for minimizing a nonlinear objective function. The algorithms are enlightened by the optimal state trajectory of an optimal control problem closely related to the minimized objective function. They are superlinear convergent when appropriate parameters are selected as required. Unlike Newton's method,  all of them can be also applied in the case of a singular Hessian matrix.  More importantly, by reduction, some of them avoid calculating the inverse of the Hessian matrix or an identical dimension matrix and some of them need only the diagonal elements of the Hessian matrix. In these cases, these algorithms still outperform the gradient descent method. The merits of the proposed optimization algorithm are illustrated by numerical experiments.

\end{abstract}
%

%
%\begin{IEEEkeywords}
%%LQR, input delays, multiplicative noise, networked control,  stochastic system.
%\end{IEEEkeywords}
\section{Introduction}

The history of optimization problems is long and spans multiple disciplines, including mathematics, operations research, computer science, and engineering \cite{rao2019engineering, belegundu2019optimization,banichuk2013introduction}. The roots of optimization can be traced back to the 18th century when mathematicians like Leonhard Euler and Joseph-Louis Lagrange began formulating and solving problems related to finding extrema (maxima or minima) of functions. Several important achievements include: the calculus of variations became a formal branch of mathematics \cite{gelfand2000calculus}; the simplex method for solving linear programming problems led to the development of operations research \cite{wolfe1959simplex}; the global optimization, dealing with finding the global optimum rather than the local optima of a nonlinear function, gained attention \cite{torn1989global}; evolutionary and metaheuristic optimization algorithms gained popularity \cite{abdel2018metaheuristic}; optimization techniques found extensive applications in operations research, engineering, finance, and various other fields. Integer programming, mixed-integer programming, and other discrete optimization techniques were developed to address real-world problems; optimization has been becoming more and more important for training models in machine learning and in various data science applications \cite{chong2023introduction}. Gradient descent and its variants, such as stochastic gradient descent, have become essential optimization techniques in the context of neural networks \cite{du2019gradient}.

The gradient descent \cite{faires2003numerical} and Newton's methods \cite{boyd2004convex} are predominant optimization methods. The choice of step size is crucial for the convergence and stability of these methods. If the step size is too large, the methods may not converge, and if it is too small, it may converge very slowly \cite{baydin2017online, fei2020new, ruder2016overview}. To control the size of each iteration and improve the convergence properties of the algorithms, the line search or backtracking techniques are utilized to determine an appropriate step size at each iteration \cite{cormen2022introduction}. It means additional calculation costs for acquiring a step size.  In practical implementations, the step size can be adjusted dynamically based on the behavior of the function. A body of variations and enhancements, such as stochastic gradient descent \cite{bottou2010large}, momentum \cite{qian1999momentum}, adaptive learning rate methods \cite{duchi2011adaptive,o2015adaptive,zeiler2012adadelta}, are thus proposed.  

Besides, Newton's method may diverge if the Hessian matrix is singular or the algorithm encounters numerical instability (when the Hessian matrix is ill-conditioned). To mitigate the aforementioned issues, quasi-Newton method \cite{broyden1967quasi}, inexact Newton method \cite{dembo1982inexact}, and modified Newton method \cite{fletcher1977modified} have been developed.  

The paper attempts to solve optimization problems (OPs) by proposing novel algorithms, which are enlightened by the optimal state trajectory of the closely related optimal control problems (OCPs). Because our original algorithm is proposed almost along the optimal state trajectory of OCP, it holds some remarkable features.  
%fast convergence, steady, and extensive application. Besides, the variants of the algorithm also inherit these good features when reducing the calculation cost. Recall that every parameter or variable in OCP has a clear physical meaning whereby it is easy to tune the parameters to alter the interested variables \cite{lewis2012optimal}. Accordingly, the convergence rate of algorithms based on OCP can be altered by adjusting those parameters inherited from OCP. 
It converges more rapidly than gradient descent because it superlinearly converges; meanwhile, it is superior to Newton's method because it can be applied in the case of a singular Hessian matrix. To reduce calculation costs, we also offer several variants of the original algorithm. It is worth pointing out that 
 the variants of the original algorithm also inherit the above merits. All of them are superlinearly convergent.

The remainder is arranged as follows. In Section II, we propose the novel algorithm almost along an optimal state trajectory.  Section III is devoted to the convergence analysis of the proposed algorithm and its several variants. The effectiveness of algorithms is illustrated in Section IV. Some conclusions are achieved in Section V.

{\bf Notation}: Throughout the paper, the superscript $T$ stands for
matrix transposition; $R^n$ denotes the $n$-dimensional real vector space; $I_n$ is $n$-dimensional unit matrix. For any square matrix $M$, $M > 0$ means
that it is positive definite and $\rho(M)$ represents the spectral radius of $M$. $|\cdot|$ and $||\cdot||$ denote appropriate vector norm and matrix norm, respectively.

%The history of optimization problems continues to evolve with advancements in mathematics, computer science, and the increasing complexity of real-world applications. It is necessary to develop new algorithms, methods, and tools to address a wide range of optimization challenges across various disciplines.

\section{Optimization algorithms enlightened by Optimal control}

Assume that $f(x): R^n \mapsto R^1$ is a twice continuously
differentiable nonlinear function. The optimization problem(OP)
\begin{align} \min_x f(x) \label{obj}\end{align}
is generally solved by numerical iteration
\begin{align}
x_{k+1}=\Phi(x_k), \mbox{a guess of }x_0  \label{phi}
\end{align}
with $\Phi(\cdot)$ being a designed function. 

We will design $\Phi(\cdot)$ in \eqref{phi} with the aid of the following optimal control problem (OPC):
\begin{align}
&\min_{u} \sum_{k=0}^N[f(x_k)+\frac{1}{2}u_k^TRu_k]+f(x_{N+1}),\label{perf}\\
&\mbox{subject to~} x_{k+1}=x_k+u_k, \label{sys1}
\end{align}
where $x_k\in R^{n}$ and $u_k\in R^{n}$ are the state and control of system \eqref{sys1}, respectively, integer $N>0$ is the control terminal time,  matrix $R>0$ is the control weight matrix,  and \eqref{perf} is the cost functional of OCP, closely related to OP \eqref{obj}.

%\begin{remark}
%It can be seen from \eqref{perf} that the optimization algorithm based on OCP will be fastest because every $u_k$ in \eqref{sys1} is chosen to minimize $\sum_{i=k+1}^{N+1}f(x_{k})$.
%\end{remark}
%
%\begin{remark} Because the control energy $\sum_{k=0}^N u_k'Ru_k$ is involved in the cost function \eqref{perf}, it can be concluded that the algorithm minimizing \eqref{perf} would be stable and always convergent.     
%the smaller the magnitude of weight matrix $R$ is, the more control energy is allowed to consume for minimization, i.e., the larger the update direction $u_k$ in \eqref{sys1} is, the more rapidly the
%iteration \eqref{sys1} may converge to the extremal point of OP \eqref{obj}.
%\end{remark}

%\begin{remark}
%OCP \eqref{perf}-\eqref{sys1} degenerates into $\min_{u_0} f(x_1)$ subject
%to \eqref{sys1} without restriction of consumed control energy.
%\end{remark}

By solving OCP \eqref{perf}-\eqref{sys1}, one can obtain the optimal control law $u_k$ and optimal state trajectory $x_k$ as below.

\begin{lemma}\label{lemanaopt}
The optimal control strategy and the optimal state trajectory of OCP \eqref{perf}-\eqref{sys1} admit
\begin{align}
u_k=&-R^{-1}\sum_{i=k+1}^{N+1}\nabla f(x_i),  \label{conlaw}\\
x_{k+1}=&x_k-R^{-1}\sum_{i=k+1}^{N+1}\nabla f(x_i),\label{impite}
\end{align}
where $\nabla f(x)$ stands for the gradient of $f(x)$.
\end{lemma}

\begin{proof}
The conclusion is direct by using the maximum principle and solving the Hamiltonian systems
\begin{align}
x_{k+1}=&x_k-R^{-1}p_k,\\
p_{k-1}=&p_k+\nabla f(x_k), p_N=\nabla f(x_{N+1}).
\end{align} Please refer to \cite{xu2023optimization} for  the detail.
\end{proof}

Along the optimal state trajectory \eqref{impite}, we will explore several algorithms for solving OP \eqref{obj}. It is feasible only if OCP \eqref{perf}-\eqref{sys1} is solvable. The fact that the global optimality means local optimality indicates that the optimal solution $u_i$ is in charge of regulating $x_{i+1}$ to minimize $\sum_{k=i+1}^N[f(x_k)+\frac{1}{2}u_k^TRu_k]+f(x_{N+1})$. Given that $R>0$, the optimal $u_i$ can minimize $\sum_{k=i+1}^{N+1}f(x_k)$  by regulating the state $x_{i+1}$. One can thereby assert that the optimal $u_N$ minimizes $f(x_{N+1})$. 

It seems impossible to exactly achieve the optimal state trajectory of \eqref{perf}-\eqref{sys1} because of the nonlinearity of $f(x)$. Consequently, we attempt to approximate the optimal state trajectory for solving \eqref{obj} by
%\indent{\rm{\bf Algorithm I:}}
\begin{align}
&x_{k+1}=x_{k}-g_k, \label{rimpite}\\
&g_k=(R+H(x_k))^{-1}(\nabla f(x_k)+\sum_{i=k+1}^{N}(\nabla f(x_i)-H(x_i)g_i), ~g_N=(R+H(x_N))^{-1}{\nabla f(x_N)},\label{gk}
\end{align}
where $H(x)$ denotes the Hessian matrix of $f(x)$. 
\begin{lemma}
For any $0\le k \le N$, $x_k$ generated by algorithm \eqref{rimpite}-\eqref{gk} is almost along the optimal state trajectory \eqref{impite}.
\end{lemma}

\begin{proof}
It can be proved by using the first-order Taylor expansion of $\nabla f(x)$ and the optimal trajectory \eqref{impite} in Lemma \ref{lemanaopt}, please refer to the proof of \cite[Lem.2]{zhang2023optimal}.
\end{proof}

Given that $x_k$ relies on $x_{i}$ for $i\ge k+1$ in \eqref{rimpite}-\eqref{gk}, it can not iterate forward, we propose
the following algorithm
 \begin{align}
x_{k+1}=&x_{k}-[I-((R+H(x_k))^{-1}R)^{N-k+1}]H(x_k)^{-1}\nabla f(x_k). \label{backite}
\end{align} 
\begin{lemma}\label{backalg}
Algorithm \eqref{backite} is obtained from \eqref{rimpite}-\eqref{gk}.   Different from any algorithm among the existing ones, algorithm \eqref{backite} recovers Newton's method when $R=0$.
\end{lemma}
 \begin{proof}
 Replace all $x_i$ for $i\ge k+1$ in the right-hand side of \eqref{rimpite} with $x_k$ and get the iterative relation
%\indent{\rm{\bf Algorithm I:}}
\begin{align}
&x_{k+1}=x_{k}-z_k(x_k), \label{xkzk}\\
&z_l(x_k)=(R+H(x_k))^{-1}(\nabla f(x_k)+Rz_{l+1}(x_k)), ~z_N(x_k)=(R+H(x_k))^{-1}{\nabla f(x_k)},\\
&l=N-1,N-2,\dots,k.\label{zk}
\end{align}
According to \eqref{zk}, $z_k(x_k)$ can be directly calculated as 
\begin{align}
z_k(x_k)=&\sum_{i=0}^{N-k}[(R+H(x_k))^{-1}R]^{i}(R+H(x_k))^{-1}\nabla f(x_k)\no\\
=&[I-((R+H(x_k))^{-1}R)^{N-k+1}]H(x_k)^{-1}\nabla f(x_k).\label{rzk}
\end{align}
By combining \eqref{xkzk} and \eqref{rzk}, it is immediate to obtain \eqref{backite}. Observe the structure of \eqref{backite}, it is indeed different from any of the existing algorithms. Let $R=0$. It is immediate to see that \eqref{backite} is reduced to Newton's method.
\end{proof}
%\begin{remark}
%It is evident that \eqref{backite} actually different from any existing variant of Newton's method.
%\end{remark}

% with a generalized Newton step size $I-((R+H(x_k))^{-1}R)^{N-k+1}$. It's called the generalized Newton step size because the standard Newton step size is a positive scalar. Due to that in Newton's method, $H(x_k)^{-1}\nabla f(x_k)$ influences how large of a step and what direction the algorithm takes in each iteration and the iteration result. To control the size of each iteration, improve the convergence properties of Newton's method, balance the aggressiveness of the algorithm with stability and efficiency considerations, and make it more versatile and applicable to a wider range of optimization problems, the Newton step size is introduced to Newton's method. It is typically chosen by solving a one-dimensional optimization problem 
%\begin{align}
%\min_{\alpha_k} f(x_k-\alpha_k H(x_k)^{-1}\nabla f(x_k)) \label{minalpha}
%\end{align}
%along the search direction given by $-H(x_k)^{-1}\nabla f(x_k)$. The minimization problem \eqref{minalpha} is typically solved using line search methods, such as backtracking line search or exact line search, which will not only yield the additional calculation cost but also reduce the square convergence speed of Newton's method to the linear convergence.
\section{Superlinear algorithms}
Lemma \ref{backalg} encourages us to explore more algorithms. 
\subsection{A superlinear algorithm}
 Note that $k$ in \eqref{backite} is required to be less than $N+1$. We thus modify \eqref{backite} as \\
  \indent{\rm{\bf Algorithm I:}}
 \begin{align}
x_{k+1}=&x_{k}-[I-((R+H(x_k))^{-1}R)^{k+1}]H(x_k)^{-1}\nabla f(x_k). \label{forwite}
\end{align} 
Only if $f(x)$ is strictly convex, there still hold $\rho(((R+H(x_k))^{-1}R)^{k+1})<1$ and $\rho(I-((R+H(x_k))^{-1}R)^{k+1})<1$. Matrix $I-((R+H(x_k))^{-1}R)^{k+1}$ slightly alters the step size and direction of Newton's method, we thus guess that \eqref{forwite} has some advantages of Newton's method, which will be proved later.

\begin{theorem}\label{thmmain0}
%Assume $\nabla f(x_*)=0$ and $H(x_*)> 0$. 
Assume $f(x)$ is strictly convex. 
Then $\{x_k\}$ generated by \eqref{forwite} admits
 \begin{align}
 |x_{k+1}-x_*|\le \rho((R+H(x_k))^{-1}R))^{k+1}|x_k-x_*|. \label{conrate2}
 \end{align}
 i.e.,  Algorithm \eqref{forwite} is superlinearly convergent.
\end{theorem}
\begin{proof}
The proof will be divided into two parts. Prove that the algorithm \eqref{forwite} is convergent and then show that it converges superlinearly.

%Let $\bar z_k(x)=[I-((R+H(x))^{-1}R)^{k+1}]H(x)^{-1}\nabla f(x)$.

First, define Lyapunov function candidate $V_k(x_k)=f(x_k)-f(x_*)$, where $x_*$ is the minimum point of $f(x)$. Then $V_k(x_k)\ge 0$ and $V_k(x_*)=0$. Along \eqref{forwite}, we can derive 
\begin{align}
&V_{k+1}(x_{k+1})-V_k(x_k)\no\\
=& f(x_{k+1})-f(x_k)\no\\
=&[\nabla f(x_k)]^T(x_{k+1}-x_{k})+o(|x_{k+1}-x_{k}|^2)\no\\
=&-[\nabla f(x_k)]^T[I-((R+H(x_k))^{-1}R)^{k+1}]H(x_k)^{-1}\nabla f(x_k)+o(|x_{k+1}-x_{k}|^2), \label{deltav}
\end{align}
Recall $\rho(((R+H(x_k))^{-1}R)^{k+1})<1$ and $H(x_k)^{-1}>0$. According to \cite[Thm.H.1.a]{marshall1979inequalities}, each eigenvalue of $(R+H(x_k))^{-1}R)^{k+1}H(x_k)^{-1}$ is positive and $0<\rho((R+H(x_k))^{-1}R)^{k+1}H(x_k)^{-1})<\rho(H(x_k)^{-1})$. Moreover, it can be derived from 
\begin{align}
&(R+H(x_k))^{-1}R)^{k+1}H(x_k)^{-1}\no\\
=&(I+R^{-1}H(x_k))^{-1})^{k+1}H(x_k)^{-1}\no\\
=&(I+R^{-1}H(x_k))^{-1}(I+R^{-1}H(x_k))^{-1}\cdots(I+R^{-1}H(x_k))^{-1}H(x_k)^{-1}\no\\
=&[H(x_k)^{-1}H(x_k)(I+R^{-1}H(x_k))^{-1}][H(x_k)^{-1}H(x_k)(I+R^{-1}H(x_k))^{-1}]\no\\
&\cdots [H(x_k)^{-1}H(x_k)(I+R^{-1}H(x_k))^{-1}]H(x_k)^{-1}\no\\
=&[H(x_k)^{-1}(I+H(x_k)R^{-1})^{-1}H(x_k)][H(x_k)^{-1}(I+H(x_k)R^{-1})^{-1}H(x_k)]\no\\
&\cdots [H(x_k)^{-1}(I+H(x_k)R^{-1})^{-1}H(x_k)]H(x_k)^{-1}\no\\
=&H(x_k)^{-1}(I+H(x_k)R^{-1})^{-1}(I+H(x_k)R^{-1})^{-1}\cdots (I+H(x_k)R^{-1})^{-1}\no\\
=&H(x_k)^{-1}R(R+H(x_k))^{-1}R(R+H(x_k))^{-1}\cdots R(R+H(x_k)^{-1})\no\\
=&H(x_k)^{-1}[R(R+H(x_k))^{-1}]^{k+1}
\end{align}
that
$[(R+H(x_k))^{-1}R)^{k+1}H(x_k)^{-1}]^{T}=(R+H(x_k))^{-1}R)^{k+1}H(x_k)^{-1}$. Consequently, $(R+H(x_k))^{-1}R)^{k+1}H(x_k)^{-1}>0$. 

Until now, it is easy to see from \eqref{deltav} that $V_{k+1}(x_{k+1})-V_k(x_k)<0$. As a consequence, \eqref{forwite} is stable and convergent.

Let $\bar z_k(x)=[I-((R+H(x))^{-1}R)^{k+1}]H(x)^{-1}\nabla f(x)$.
%\begin{align}
%\bar z_k=&[I-(I-MH(x))^{k+1}](MH(x))^{-1}M\nabla f(x)\no\\
%=&[I-(I-MH(x))^{k+1}]H(x)^{-1}\nabla f(x)
%\end{align}
Now direct calculation gives
\begin{align}
\bar z_k(x_*)=&[I-((R+H(x_*))^{-1}R)^{k+1}]H(x_*)^{-1}\nabla f(x_*)=0, \label{barzxstar}\\
\bar z_k'(x_*)=&I-((R+H(x_*))^{-1}R)^{k+1}. \label{debarzxstar}
\end{align}
From \eqref{forwite}, 
\begin{align}A
&x_{k+1}-x_*\no\\
=&x_k-x_*-\bar z_k(x_k)\no\\
\approx&x_k-x_*-[\bar z_k(x_*)+\bar z_k'(x_*)(x_k-x_*)]\no\\
=&x_k-x_*-\bar z_k'(x_*)(x_k-x_*)\no\\
=&(I-\bar z_k'(x_*))(x_k-x_*)\no\\
=&((R+H(x_*))^{-1}R)^{k+1}(x_k-x_*),
\end{align}
where the third line has used the Taylor expansion of $\bar z_k(x)$ at $x_*$. 
Due to that $\rho((R+H(x_*))^{-1}R)<1$ when $H(x_*)>0$, it is immediate to conclude that \eqref{conrate2} holds. The proof is completed.
%Similar to Lemma \ref{lembargk}, one can show that $\hat g'_k(x_*)=(\frac{R}{R+H(x_*)})^{k+1}$.
\end{proof}

\subsection{Superlinear algorithm without involving the inverse of Hessian matrix}

%Recently, based on the optimal control problem, \cite{zhang2023csi} proposed the algorithm below
% \begin{align}
%x_{k+1}=&x_{k}-\hat g_k(x_k), \label{itewithhatgk}\\
%\hat g_k(x_k)=&(R+H(x_k))^{-1}[\nabla f(x_k)+R \hat g_{k-1}(x_k)], \hat g_0(x_k)=(R+H(x_k))^{-1}\nabla f(x_k). \label{hatgk}
%\end{align}
%The above algorithm is appealing because of its stability, extensive application, and superlinear convergence. It has some advantages of Newton's method but abandons its several disadvantages.
%
%In fact, algorithm \eqref{itewithhatgk}-\eqref{hatgk} can be reformulated as 
% \begin{align}
%x_{k+1}=&x_{k}-\{I-[(R+H(x_k))^{-1}R]^{k+1}\}H(x_k)^{-1}\nabla f(x_k). \label{ritewithhatgk}
%\end{align}
%If we regard $I-[(R+H(x_k))^{-1}R]^{k+1}$ as a Newton step size of \eqref{ritewithhatgk}, it is evident that it is unnecessary to solve another optimization problem to acquire such a Newton step size. More importantly, \eqref{ritewithhatgk} is superlinearly convergent.

It cannot be denied that in \eqref{forwite},  the calculation cost focuses on the inverse of $R+H(x_k)$. The good news is that $R$ in \eqref{forwite} is an adjustable matrix, so is matrix $R+H(x_k)$. To save calculation, 
we replace $(R+H(x_k))^{-1}$ with an adjustable matrix $M$. Note that $(R+H(x_k))^{-1}R=I-(R+H(x_k))^{-1}H(x_k)$. $(R+H(x_k))^{-1}R$ can thus be substituted by $I-MH(x_k)$. 
Then algorithm \eqref{forwite} can be modified as \\
\indent{\rm{\bf Algorithm II:}}
\begin{align}
x_{k+1}=x_k-[I-(I-MH(x_k))^{k+1}]H(x_k)^{-1}\nabla f(x_k). \label{forwitem}
\end{align}

Algorithm \eqref{forwitem} inherits the superlinear convergence of algorithm \eqref{forwite}, which will be shown in the sequel.

\begin{theorem}\label{thmmain}
Assume $\nabla f(x_*)=0$ and $H(x_*)> 0$. If $M>0$ and $\rho(I-MH(x_*))<1$, then $\{x_k\}$ generated by \eqref{forwitem} admits
 \begin{align}
 |x_{k+1}-x_*|\le \rho(I-MH(x_*))^{k+1}|x_k-x_*|. \label{conrate3}
 \end{align}
 i.e.,  Algorithm \eqref{forwitem} is superlinearly convergent.
\end{theorem}
\begin{proof}
Denote 
\begin{align}
\hat g_k(x)=[I-(I-MH(x))^{k+1}]H(x)^{-1}\nabla f(x).
\end{align}
Direct derivation shows
\begin{align}
\hat g_k(x_*)=&[I-(I-MH(x_*))^{k+1}]H(x_*)^{-1}\nabla f(x_*)=0, \label{hatgkxstar}\\
\hat g'_k(x_*)=&I-(I-MH(x_*))^{k+1}. \label{dehatgkxstar}
\end{align}
From \eqref{forwitem}, 
\begin{align}
&x_{k+1}-x_*\no\\
=&x_k-x_*-\hat g_k(x_k)\no\\
\approx&x_k-x_*-[\hat g_k(x_*)+g'_k(x_*)(x_k-x_*)]\no\\
=&x_k-x_*-\hat g'_k(x_*)(x_k-x_*)\no\\
=&(I_n-\hat g'_k(x_*))(x_k-x_*)\no\\
=&(I-MH(x_*))^{k+1}(x_k-x_*).
\end{align}
In view of $ \rho(I-MH(x_*))<1$ when $H(x_*)>0$, the proof is completed.
%Similar to Lemma \ref{lembargk}, one can show that $\hat g'_k(x_*)=(\frac{R}{R+H(x_*)})^{k+1}$.
\end{proof}

\begin{remark}
It should be pointed out that, unlike Newton's method,  algorithm \eqref{forwitem} still works even if there occasionally holds $H(x)=0$ because algorithm \eqref{forwitem} has an equivalent realization
 \begin{align}
x_{k+1}=&x_{k}-\hat g_k(x_k), \label{itewithhatgkm}\\
\hat g_l(x_k)=&M\nabla f(x_k)+(I-MH(x_k)) \hat g_{l-1}(x_k), \hat g_0(x_k)=M\nabla f(x_k),l=1,\dots,k. \label{hatgkm}
\end{align}
\end{remark}

\begin{remark} When $H(x)>0$ and $M>0$, only if the spectral radius of $M$ is small enough, $\rho(I-MH(x))<1$ always holds by appropriately selecting $M$. 
\end{remark}

This remark makes us realize that it is possible to select a diagonal matrix $M$ in algorithm \eqref{itewithhatgkm}-\eqref{hatgkm} to ensure that the algorithm is superlinearly convergent. 

\subsection{Superlinear algorithms without involving Hessian matrix}

In addition, to alleviate the requirement for the second-order information,  the Hessian matrix $H(x_k)$ in algorithm \eqref{itewithhatgkm}-\eqref{hatgkm} can be approximated by the first-order difference. As a result, algorithm \eqref{itewithhatgkm}-\eqref{hatgkm} is replaced by \\
\indent{\rm{\bf Algorithm III:}}
 \begin{align}
x_{k+1}=&x_{k}-\bar g_k(x_k), \label{diaite}\\
\bar g_l(x_k)=&D \nabla f(x_k)+(I-D D_1(x_k)) \bar g_{l-1}(x_k), \bar g_0(x_k)=D \nabla f(x_k),l=1,\dots,k, \label{bargk}
\end{align}
where $D>0$ is an appropriate diagonal matrix and $D_1(x_k)$ is the first-order backward difference of $\nabla f(x)$ at point $x_k$. Denote $x_{k,i}$ as the $i-th$ element of $x_k$ and $\nabla f_i(x)$ as the $i-th$ element of $\nabla$. Then the $i,j-th$ element of $D_1(x_k)$ is given by $[D_1(x_k)]_{i,j}=[(\nabla f_i(x_k)-\nabla f_i(x_{k-1}))/(x_{j,k}-x_{j,k-1})]$.

To reduce calculation cost further, we replace the Hessian matrix $H(x_k)$ with a diagonal matrix $\Lambda(x_k)$. It has the identical diagonal elements with the Hessian matrix $H(x_k)$. That is to say, $\Lambda(x_k)={\rm diag}\{\frac{\partial f^2(x)}{\partial x_1^2(x)},\cdots, \frac{\partial f^2(x)}{\partial x_n^2}\}|_{x=x_k}$. As a result, algorithm \eqref{itewithhatgkm}-\eqref{hatgkm} is modified as \\
\indent{\rm{\bf Algorithm IV:}}
 \begin{align}
x_{k+1}=&x_{k}-\tilde g_k(x_k), \label{dualdiaite}\\
\tilde g_l(x_k)=&D \nabla f(x_k)+(I-D \Lambda(x_k)) \tilde g_{l-1}(x_k), \tilde g_0(x_k)=D \nabla f(x_k),l=1,\dots,k, \label{tildegk}
\end{align}
where $D>0$ is an appropriate diagonal matrix.

\begin{remark}
Only if we select $D$ such that $I-D \Lambda(x_k)>0$ and $\rho(I-D \Lambda(x_k))<1$, then Algorithm IV is superlinearly convergent.
\end{remark}

From the above discussion, the computational complexity of Algorithm I, Algorithm II, Algorithm III, and Algorithm IV has been reduced sequentially.

%\begin{theorem}\label{thm2}
%Assume $\nabla f(x_*)=0$ and $H(x_*)> 0$. Consider the iteration scheme
%\begin{align}
%x_{k+1}=&x_{k}-\tilde g_{k}(x_k), \label{fexpiten}\\
%\tilde g_{k}(x)=&\frac{1}{R+H(x)}[\nabla f(x)+R\tilde g_{k-1}(x)], \tilde g_0=\frac{\nabla f(x)}{R+H(x)}. \label{tildegk}
%\end{align}
% Then \eqref{fexpiten}-\eqref{tildegk} superlinearly convergent to $x_*$, i.e.,
% \begin{align}
% |x_k-x_*|\le \prod_{i=1}^{k+1}(\frac{R}{R+H(x_*)})^{i+1}|x_0-x_*|.
% \end{align}
%\end{theorem}
%\begin{proof}
%Along similar reasoning as the proof of  Lemma \ref{thm1}, the result is straightforward.
%\end{proof}

%\begin{theorem}
%The iteration scheme
%\begin{align}
%x_{k+1}=x_k-R^{-1}\sum_{i=k+1}^{N+1} \nabla f(x_i)
%\end{align}
%
%\end{theorem}

\section{Conclusion}
In the paper, several optimization algorithms have been developed. They are enlightened by the optimal trajectory of the related OCP. The algorithms converge more rapidly than the gradient descent. Meanwhile, unlike Newton's method, they still work even if the Hessian matrix in iterate is singular or ill-conditioned. The convergence rate of the algorithms will be superlinear if the adjustable matrices are selected as required.

 \bibliographystyle{ieeetr}
\bibliography{reference}

\end{document}